\documentclass{fundam}

\usepackage[T1]{fontenc}
\usepackage{amsmath,amssymb,mathtools}
\usepackage{booktabs}
\usepackage{enumitem}
\usepackage{url}
\newtheorem{problem}[definition]{Problem}

\newcommand{\sfw}{\mathrm{sfw}}
\newcommand{\lsfw}{\mathrm{lsfw}}
\newcommand{\lrw}{\mathrm{lrw}}
\newcommand{\rw}{\mathrm{rw}}
\newcommand{\pw}{\mathrm{pw}}
\newcommand{\nw}{\mathrm{nw}}
\newcommand{\anw}{\mathrm{anw}}
\newcommand{\crk}{\mathrm{cr}}
\newcommand{\GF}{\mathbb F}
\newcommand{\Boxop}{\mathbin{\Box}}

\title{Split-Free Cable Expressions: Active Neighborhood Profiles and Linear Rank-Width}

\author{Antonios Kalampakas\corresponding\\
Department of Mathematics, College of Engineering\\
American University of the Middle East\\
Egaila 54200, Kuwait\\
antonios.kalampakas{@}aum.edu.kw}

\begin{document}

\address{Antonios Kalampakas, Department of Mathematics, College of Engineering,
American University of the Middle East, Egaila 54200, Kuwait.
E-mail: \url{antonios.kalampakas@aum.edu.kw}}

\maketitle

\runninghead{A. Kalampakas}{Split-Free Cable Expressions and Active Neighborhood Profiles}

\begin{abstract}
We introduce split-free cable expressions and their sequential
restriction. Live cables are vertex blocks that future operations cannot
split. The main result identifies sequential split-free cable width, up
to an additive two, with active neighborhood-width, the minimum number of
distinct nonzero future-neighborhood classes across a linear layout. The
lower bound extracts such a layout from every cable play and includes a
parity argument for prefixes inside a birth group. The upper bound
compiles any active-neighborhood layout into a singleton-birth play. This
gives an exponential upper bound and a linear lower bound in terms of
linear rank-width. We prove that the complement of the seven-vertex cycle
has linear rank-width two and sequential cable width exactly eight. This
disproves the proposed affine upper estimate already at rank two. We also
prove an unbounded separation between branching expression width and
sequential width on trees.
\end{abstract}


\noindent\emph{ACM Computing Classification System (2012):}
Mathematics of computing, Discrete mathematics, Graph theory.

\section{Introduction}\label{sec:intro}

Rank-width measures a graph through the ranks over $\GF_2$ of the cuts
displayed by a branch-decomposition. Linear rank-width restricts the
decomposition to a linear order of the vertices and takes the largest
cutrank over its prefixes. These parameters were introduced by Oum and
Seymour \cite{Oum05,OS06}. They connect dense graph structure with
clique-width, vertex-minors, and decomposition algorithms
\cite{HO08,Oum17}. In particular,
\[
        \rw(G)\le \operatorname{cw}(G)
        \le 2^{\rw(G)+1}-1.
\]
This connection supports algorithmic methods based on labelled
expressions and logical meta-theorems \cite{CMR00}.

Linear expression parameters approach the same path-like structure from
the construction side. Linear NLC-width and linear clique-width build a
graph along one spine \cite{GW05}. Neighborhood-width instead measures
how many different future neighborhoods occur on the born side of a
linear layout \cite{Gurski06}. It differs from both linear NLC-width and
linear clique-width by at most one. Cutrank records the dimension of the
same family of rows, while neighborhood-width records how many row types
actually occur. The distinction between dimension and number of types is
central here.

We introduce a typed graph construction in which a live interface
consists of \emph{cables}. A cable is a vertex block that every later
operation must treat uniformly. A birth creates a cable, a block joins
two cables completely, a clique fills one cable, a fold merges cables,
and a close removes a cable whose incident edges are complete. The
calculus has no split operation. Its unrestricted expressions may use
parallel composition, while its sequential plays place the same
operations on one spine. This gives the split-free width $\sfw$ and the
sequential split-free width $\lsfw$.

The main comparison uses a small modification of neighborhood-width.
For a prefix of a vertex order, ignore the zero row and count the distinct
nonzero rows of the prefix-to-suffix adjacency matrix. Minimizing the
largest such count over all orders defines the active neighborhood-width
$\anw$. It differs from the standard neighborhood-width $\nw$ by at
most one. We prove
\[
        2\,\anw(G)\le \lsfw(G)\le 2\,\anw(G)+2.
\]
The lower bound is stronger than a cutrank estimate. In the birth order
of a play of width $w$, every prefix has at most
$\lfloor w/2\rfloor$ distinct nonzero rows. Prefixes that end at birth
boundaries follow directly from cable uniformity. A prefix inside a
birth group can have one provisional extra row type. At odd width, the
next legal move is forced to be a fold or a close, and that move removes
the extra type.

The upper bound is constructive. Starting from a layout with at most
$k$ active row types, keep one resident cable for each current nonzero
row class. Birth the next vertex as a singleton and join it to the
resident classes that it sees. Moving the boundary deletes one matrix
column. Old row classes can therefore merge or become zero, but they
cannot split. Folds and closes restore the resident family for the next
prefix. The resulting play has width at most $2k+2$.

Since a rank-$k$ binary matrix has at most $2^k-1$ nonzero rows, the
calibration yields
\[
        2\,\lrw(G)\le \lsfw(G)
        \le 2^{\lrw(G)+1}.
\]
Thus sequential cable width is controlled on every class of bounded
linear rank-width. The lower and upper proofs also explain why the
number of row profiles, rather than their dimension alone, is the natural
state space for the construction.

The exponential upper bound cannot in general be replaced by the affine
estimate $2\lrw(G)+2$. For the complement $H$ of the seven-vertex cycle
we prove
\[
        \lrw(H)=2,\qquad \anw(H)=3,\qquad
        \lsfw(H)=8.
\]
The graph is twin-free, so every cable play births its vertices
separately. Immediately after the fourth birth in a width-seven play,
four cables would be live. The next move would have to be a fold or a
close before any new edge could be placed. The cyclic structure rules out
every such move.

The branching and sequential variants also behave differently. Every
tree has an unrestricted split-free expression of width at most four. On
the other hand, linear rank-width equals path-width on forests
\cite{AK15}, so the general lower bound gives
$\lsfw(T)\ge 2\pw(T)$. Trees of unbounded path-width therefore give an
unbounded separation between $\sfw$ and $\lsfw$.

Section~\ref{sec:definitions} defines layouts, active
neighborhood-width, cable expressions, and cable plays.
Section~\ref{sec:calibration} proves the two-sided calibration.
Section~\ref{sec:branching} separates branching expressions from
sequential plays on trees. Section~\ref{sec:ranktwo} studies rank-two row
profiles and proves the seven-vertex counterexample.
Section~\ref{sec:conclusion} records the remaining questions.

\section{Layouts and split-free cable constructions}\label{sec:definitions}

All graphs are finite, simple, and undirected. We assign value zero to
all four width parameters on the empty graph. For a graph $G$ and
$X\subseteq V(G)$, the cutrank $\rho_G(X)$ is the rank over $\GF_2$ of
the $X$ by $V(G)\setminus X$ adjacency matrix. Let
\[
        \pi=(v_1,\ldots,v_n)
\]
be a linear layout of $G$, and put
$P_i=\{v_1,\ldots,v_i\}$ for $0\le i\le n$. The cutrank width of $\pi$
is $\max_i\rho_G(P_i)$. The linear rank-width $\lrw(G)$ is the minimum
of this quantity over all layouts. Rank-width $\rw(G)$ is defined in the
same way from the cuts displayed by a subcubic branch-decomposition, and
$\rw(G)\le\lrw(G)$.

For $u\in P_i$, write
\[
        r_i^\pi(u)=N_G(u)\cap(V(G)\setminus P_i).
\]
This set is the row of $u$ in the prefix cut matrix. Define
\[
\begin{aligned}
 \mathcal N_i^\pi
    &=\{r_i^\pi(u):u\in P_i\},\\
 \mathcal A_i^\pi
    &=\mathcal N_i^\pi\setminus\{\emptyset\}.
\end{aligned}
\]
For $P_0=\emptyset$, both families are empty. The neighborhood-width of
the layout and its active version are
\[
        \nw_\pi(G)=\max_i|\mathcal N_i^\pi|,
        \qquad
        \anw_\pi(G)=\max_i|\mathcal A_i^\pi|.
\]
The graph parameters $\nw(G)$ and $\anw(G)$ are obtained by minimizing
over all layouts. The first is the standard neighborhood-width
\cite{Gurski06}. The second removes the zero class, whose vertices have
no neighbor in the suffix.

\begin{lemma}[active and standard neighborhood-width]\label{lem:nwcompare}
For every graph $G$,
\[
        \anw(G)\le\nw(G)\le\anw(G)+1
\]
and
\[
        \lrw(G)\le\anw(G)\le 2^{\lrw(G)}-1.
\]
\end{lemma}

\begin{proof}
For every layout and every prefix, the family of all rows contains the
nonzero rows and at most one additional zero row. Minimizing first an
optimal layout for each side gives
$\anw(G)\le\nw(G)\le\anw(G)+1$.

The rank of a matrix is at most the number of its distinct nonzero rows.
This gives $\lrw(G)\le\anw(G)$. Conversely, choose a layout of cutrank
at most $k=\lrw(G)$. At every prefix its rows lie in a binary vector
space of dimension at most $k$, which has at most $2^k-1$ nonzero
vectors. Hence $\anw(G)\le2^k-1$.
\end{proof}

We now define the construction language. A \emph{cable} is a nonempty
set of vertices treated as one indivisible live unit. A split-free cable
expression is a typed term generated by the following operations. A type
$p\to q$ receives $p$ ordered live cables and returns $q$ ordered live
cables. Its cable rank is
\[
        \crk(p\to q)=p+q.
\]
Composition is chronological, so $s\circ t$ means that $s$ is performed
before $t$.

\begin{center}
\begin{tabular}{lll}
\toprule
symbol & type & effect \\
\midrule
$e$ & $1\to1$ & keep one cable unchanged \\
$\nu_S$ & $0\to1$ & birth a fresh vertex set $S$ as one cable \\
$\kappa$ & $1\to0$ & close one cable \\
$\varphi$ & $2\to1$ & fold two cables into their union \\
$B$ & $2\to2$ & add all edges between two cables \\
$K$ & $1\to1$ & add all internal edges of one cable \\
\bottomrule
\end{tabular}
\end{center}

If $s:p\to q$ and $t:q\to r$, then
$s\circ t:p\to r$. If $s:p\to q$ and $t:p'\to q'$, then
\[
        s\Boxop t:p+p'\to q+q'.
\]
The parallel product executes the two terms on disjoint vertex sets and
concatenates their output cables. Cable permutations are structural
reindexings of the ordered interface. They create no vertices or edges
and do not change cable rank.

The value of a closed expression $t:0\to0$ is the graph obtained from
its block and clique operations. Edge additions use set union. An
expression is legal for a target graph $G$ if every vertex is born
exactly once, every added edge belongs to $G$, and the final graph is
$G$. Its width is the largest cable rank of a subterm occurrence. The
split-free width is
\[
        \sfw(G)=
        \min_t\max\{\crk(u):u\text{ is a subterm occurrence of }t\},
\]
where $t$ ranges over legal closed expressions for $G$. The name
\emph{split-free} records the defining restriction. Once two vertices
have entered one cable, no operation can separate them.
The rank of a subterm is its own interface type in the parse tree. In an
expression $e\Boxop s$, the enclosing occurrence pays for the additional
identity wire, but the internal subterms of $s$ retain their own types
and cable ranks.

\begin{remark}[relation with typed graph algebras]\label{rem:graphoid}
The serial and parallel interface syntax follows the graphoid viewpoint
of Bozapalidis and Kalampakas \cite{BK04}. In a graphoid, a type records
external boundary data. Here each interface position instead records a
live vertex block with a future-uniformity constraint. The cable widths
are new parameters attached to this split-free interface semantics. The
rank-oriented use of graph operations may also be compared with the
algebraic characterization of rank-width by Courcelle and Kant\'e
\cite{CK09}.
\end{remark}

The sequential form of the calculus is a cable play. A configuration on
$G$ consists of a family $\mathcal L$ of pairwise disjoint live cables
among the born vertices and a set of already placed edges. Let
$\ell=|\mathcal L|$ before a move.

\begin{definition}[sequential cable play]\label{def:play}
A play on $G$ uses the following moves.
\begin{itemize}[leftmargin=2em]
\item $\mathrm{birth}(S)$ introduces a nonempty set of fresh vertices as
one cable. Its cost is $2\ell+1$.
\item $\mathrm{fold}(X,Y)$ replaces two live cables by $X\cup Y$.
Its cost is $2\ell-1$.
\item $\mathrm{block}(X,Y)$ places every edge between two live cables.
It is legal when all these pairs are edges of $G$ and at least one has
not already been placed. Its cost is $2\ell$.
\item $\mathrm{clique}(X)$ places every internal edge of a live cable.
It is legal when $X$ is a clique of $G$ and at least one internal edge
has not already been placed. Its cost is $2\ell$.
\item $\mathrm{close}(X)$ removes a live cable. It is legal when every
edge of $G$ incident with a vertex of $X$ has already been placed.
Its cost is $2\ell-1$.
\end{itemize}
A play constructs $G$ if every vertex is born, every edge is placed, and
every cable is closed. The sequential split-free width $\lsfw(G)$ is the
minimum possible maximum cost of such a play.
\end{definition}

A play is a cable expression whose parse tree has one spine. At each
move, the local letter is placed in parallel with identities on the other
live cables. The displayed costs are exactly the cable ranks of these
letters in context. Consequently
\[
        \sfw(G)\le\lsfw(G).
\]

Two elementary facts will be used throughout.

\begin{lemma}[birth rigidity]\label{lem:birthrigidity}
Vertices born in one cable are false twins or true twins in the final
graph. In particular, they have the same neighborhood outside their
birth cable.
\end{lemma}

\begin{proof}
The vertices are never separated. Every later block reaches all of them
or none of them, and every later clique contains all of them or none of
them. A clique either supplies every internal edge of the birth cable or
no such edge is ever supplied.
\end{proof}

\begin{lemma}[monotonicity]\label{lem:monotonicity}
If $H$ is an induced subgraph of $G$, then
$\sfw(H)\le\sfw(G)$ and $\lsfw(H)\le\lsfw(G)$. On a disconnected
graph, each width is the maximum of its values on the connected
components.
\end{lemma}

\begin{proof}
Restrict a legal construction to $V(H)$ and delete operations that become
vacuous. Empty restricted cables are suppressed, and a fold with only
one nonempty input becomes an identity. This does not increase any
interface or play cost. Components may be constructed independently in
parallel for expressions and consecutively for plays. The reverse
inequalities follow because every component is induced.
\end{proof}

\section{Calibration by active neighborhood-width}\label{sec:calibration}

This section proves the central two-sided comparison. The lower bound
extracts a layout from a play. The upper bound turns a layout back into a
play. We begin with the uniformity that a live cable imposes across the
current cut.

\begin{lemma}[cable uniformity]\label{lem:uniformity}
Fix a moment of a play, a cable $X$ live at that moment, and a vertex $v$
not yet born. In the final graph, $v$ is adjacent to all vertices of
$X$ or to none of them.
\end{lemma}

\begin{proof}
The vertices of $X$ remain together until they close, possibly inside a
larger cable created by folds. Every later block or clique involving one
member of $X$ therefore involves every member.
\end{proof}

\begin{lemma}[closed rows]\label{lem:closedrows}
If a cable is closed at some moment, its vertices have no neighbors among
the vertices unborn at that moment.
\end{lemma}

\begin{proof}
A close is legal only after every incident edge has been placed. An edge
to an unborn vertex cannot already have been placed.
\end{proof}

At a prefix that ends between two births, these lemmas give at most one
nonzero row type per live cable. A prefix inside a birth group may also
see the row of the part of that group already placed in the order. Odd
width forces this provisional extra type to collapse.

\begin{lemma}[odd-width collapse]\label{lem:oddcollapse}
Suppose a play of width $2k+1$ births a set $S$ while $k$ cables are live.
Let $B$ be the vertices born earlier and let
$\emptyset\ne S'\subsetneq S$. The cut matrix with row side
$B\cup S'$ has at most $k$ distinct nonzero rows.
\end{lemma}

\begin{proof}
Immediately after the birth there are $k+1$ live cables. Another birth
would cost $2k+3$, while a block or clique would cost $2k+2$.
The next move is therefore a fold or a close, and no edge incident with
$S$ is placed before it.

Before using this forced move, Lemmas~\ref{lem:uniformity},
\ref{lem:closedrows}, and \ref{lem:birthrigidity} give at most one row
type for each old cable and one for $S'$. There are at most $k+1$
provisional nonzero types. The columns of the cut are
\[
        (S\setminus S')\cup
        \{\text{vertices born after }S\}.
\]
We show that the forced move makes one provisional type zero or identifies
two of them.

If an old cable closes, it has no neighbor in $S$ or among later vertices,
so its row is zero. If $S$ closes, the row of $S'$ is zero.

Suppose $S$ folds with an old cable $X$. On vertices born later, $X$ and
$S'$ have the same row by cable uniformity after the fold. It remains to
compare their entries on $S\setminus S'$. By birth rigidity, $S$ is
independent or a clique. If it is independent, no later clique may act
on a descendant cable containing $S$, so no edge between $X$ and $S$ can
be created after the fold. If it is a clique, its internal edges must be
created by a later clique on a descendant cable. Legality then forces
every edge between $X$ and $S$. In both cases the row of $X$ agrees with
the row of $S'$ on $S\setminus S'$.

Finally, suppose two old cables $X$ and $Y$ fold. Their entries agree on
vertices born later by uniformity after the fold. No edge incident with
$S$ was placed before the fold. Every later operation that can create an
edge from $S$ to $X$ or $Y$ treats the common descendant of $X$ and $Y$
uniformly. Their entries therefore also agree on $S\setminus S'$.

Each possible forced move removes one provisional nonzero type. At most
$k$ distinct nonzero rows remain.
\end{proof}

\begin{proposition}[a play yields an active-neighborhood layout]
\label{prop:playtolayout}
The birth order of every play of width $w$ is a layout of active
neighborhood-width at most $\lfloor w/2\rfloor$. Within each birth group,
the vertices may be ordered arbitrarily.
\end{proposition}

\begin{proof}
Fix the birth order and a prefix. First suppose the prefix ends before a
birth. If $\ell$ cables are live immediately before that birth, then
$2\ell+1\le w$. Closed vertices have zero rows by
Lemma~\ref{lem:closedrows}, and every live cable gives one row type by
Lemma~\ref{lem:uniformity}. The number of nonzero rows is at most
\[
        \ell\le\left\lfloor\frac{w-1}{2}\right\rfloor
        \le\left\lfloor\frac w2\right\rfloor.
\]

Now let the prefix cut through a birth group $S$. Write it as
$B\cup S'$ with $\emptyset\ne S'\subsetneq S$, and let $\ell$ be the
number of old live cables before the birth. The old live cables give at
most $\ell$ row types. Closed vertices give the zero row. The vertices
of $S'$ give at most one further type by birth rigidity. Thus there are
at most $\ell+1$ nonzero rows.

If $w=2k$, then $\ell\le k-1$, so $\ell+1\le k$. If $w=2k+1$ and
$\ell<k$, the same bound holds. The remaining case has $w=2k+1$ and
$\ell=k$, where Lemma~\ref{lem:oddcollapse} gives at most $k$ nonzero
rows. Every prefix therefore has at most $\lfloor w/2\rfloor$ active
row types.
\end{proof}

The converse construction is a canonical sweep of the row classes.

\begin{proposition}[an active-neighborhood layout yields a play]
\label{prop:layouttoplay}
If a layout $\pi$ has active neighborhood-width at most $k$, then $G$
has a singleton-birth play of width at most $2k+2$.
\end{proposition}

\begin{proof}
Let $\pi=(v_1,\ldots,v_n)$ and $P_i=\{v_1,\ldots,v_i\}$. After
processing $P_i$, maintain one live cable for each nonzero row class of
the cut $(P_i,V(G)\setminus P_i)$. Also maintain that every edge of
$G[P_i]$ has been placed. The invariant is empty at $P_0$.

Assume it holds for $P_i$ and set $v=v_{i+1}$. There are at most $k$
resident cables. Birth $v$ as a singleton, at cost at most $2k+1$.
Every resident cable is one row class at $P_i$. Since the column of $v$
is constant on a row class, $N(v)\cap P_i$ is a union of resident
cables. Apply one block from $\{v\}$ to each resident cable in this
union. Every such block costs at most $2k+2$, and these blocks place
exactly the edges from $v$ to $P_i$. Hence every edge of
$G[P_{i+1}]$ is now placed.

Moving from $P_i$ to $P_{i+1}$ deletes the column of $v$ from every old
row. Equal old rows remain equal, so no old cable needs to split.
Some old row classes may become zero or merge. The singleton $\{v\}$
contributes the only new row. Close each cable whose new row is zero.
This is legal because it has no neighbor outside $P_{i+1}$ and all its
edges inside $P_{i+1}$ have been placed. Fold cables whose new nonzero
rows are equal until one cable remains for each nonzero class. Folds and
closes occur with at most $k+1$ live cables and cost at most $2k+1$.
The invariant is restored.

At $P_n$, every row is zero, so the remaining cables close legally.
The maximum cost is at most $2k+2$.
\end{proof}

\begin{theorem}[active neighborhood-width calibration]
\label{thm:calibration}
For every graph $G$,
\[
        2\,\anw(G)\le\lsfw(G)\le2\,\anw(G)+2.
\]
\end{theorem}

\begin{proof}
For the lower bound, take an optimal play of width $w$. By
Proposition~\ref{prop:playtolayout},
$\anw(G)\le\lfloor w/2\rfloor$, so $2\anw(G)\le w$.
For the upper bound, apply Proposition~\ref{prop:layouttoplay} to a layout
realizing $\anw(G)$.
\end{proof}

\begin{remark}[sharp endpoints]\label{rem:sharp}
Both endpoints can occur. Every nontrivial complete graph has
$\anw(K_n)=1$ and a width-two play obtained by one birth, one clique,
and one close. Hence $\lsfw(K_n)=2$. The graph $H$ in
Section~\ref{sec:ranktwo} has $\anw(H)=3$ and $\lsfw(H)=8$, so it
attains the upper endpoint.
\end{remark}

The theorem immediately places sequential cable width among the standard
linear layout parameters.

\begin{corollary}[comparison with neighborhood-width]
\label{cor:nw}
For every graph $G$,
\[
        2\,\nw(G)-2\le\lsfw(G)\le2\,\nw(G)+2.
\]
\end{corollary}

\begin{proof}
Combine Theorem~\ref{thm:calibration} with
Lemma~\ref{lem:nwcompare}.
\end{proof}

\begin{corollary}[comparison with linear rank-width]
\label{cor:lrw}
For every graph $G$,
\[
        2\,\lrw(G)\le\lsfw(G)\le2^{\lrw(G)+1}.
\]
In particular, $\lrw(G)\le1$ implies $\lsfw(G)\le4$, and
$\lrw(G)\le2$ implies $\lsfw(G)\le8$.
\end{corollary}

\begin{proof}
The lower bound follows from
$\lrw(G)\le\anw(G)$ and Theorem~\ref{thm:calibration}. For the upper
bound, put $k=\lrw(G)$. Lemma~\ref{lem:nwcompare} and the same theorem
give
\[
        \lsfw(G)\le2(2^k-1)+2=2^{k+1}.
\]
\end{proof}

\begin{remark}[dimension and row types]\label{rem:dimensiontypes}
Linear rank-width measures the dimension of the prefix row space.
Active neighborhood-width measures the number of nonzero vectors from
that space that actually occur. The gap between the linear and
exponential bounds in Corollary~\ref{cor:lrw} is exactly the possible gap
between these two quantities. The canonical sweep keeps the occurring
row classes themselves. This is why folds are always safe when a column
deletion makes two classes equal.
\end{remark}

\section{Branching expressions and sequential separation}\label{sec:branching}

Parallel composition lets an unrestricted expression build a subtree
independently and expose only its attachment vertex. This keeps the
expression width bounded on all trees.

\begin{theorem}[trees have bounded expression width]\label{thm:trees}
Every tree $T$ satisfies
\[
        \sfw(T)\le4.
\]
\end{theorem}

\begin{proof}
Root $T$ at a vertex $r$. We construct by induction an expression
$t(T,r):0\to1$ of width at most four. Its only output cable is
$\{r\}$, and all other vertices have closed.

For a one-vertex tree, take $t(T,r)=\nu_{\{r\}}$. Let
$r_1,\ldots,r_d$ be the children of $r$, and let $T_i$ be the subtree
rooted at $r_i$. Start with $\nu_{\{r\}}$. For each child, compose
with
\[
        (e\Boxop t(T_i,r_i))
        \circ B
        \circ(e\Boxop\kappa).
\]
The parallel expression builds $T_i$ beside the live root cable. The
block places the edge $rr_i$, and the close removes the child root. The
expression $e\Boxop t(T_i,r_i)$ has type $1\to2$, the block has type
$2\to2$, and the padded close has type $2\to1$. Their cable ranks are
at most four. The induction gives the same bound inside each child
expression. Closing the final root produces a legal closed expression
for $T$.
\end{proof}

\begin{corollary}[unbounded sequential gap]\label{cor:treegap}
The difference between $\lsfw$ and $\sfw$ is unbounded, even on trees.
More precisely, every forest $F$ satisfies
\[
        \lsfw(F)\ge2\,\pw(F),
\]
while every tree has expression width at most four.
\end{corollary}

\begin{proof}
Linear rank-width equals path-width on forests \cite{AK15}. The lower
bound follows from Corollary~\ref{cor:lrw}. Trees have unbounded
path-width, while Theorem~\ref{thm:trees} bounds their expression width.
\end{proof}

This separation is structural. A branching expression may finish a child
subtree and close its interface before starting another child. A
sequential play must choose one global birth order, and its active
neighborhood classes record the path-like congestion of that choice.

\section{Rank-two profiles and failure of the affine estimate}
\label{sec:ranktwo}

For a layout of cutrank at most two, every prefix row belongs to a binary
space of dimension at most two. At most three nonzero row types can
occur. The general sweep therefore uses at most three resident cables
and one transient arrival cable, giving width eight. We first isolate the
condition under which two residents suffice.

\begin{lemma}[rank-two cut anatomy]\label{lem:ranktwo}
Let $P$ be a prefix with $\rho_G(P)\le2$. If the cutrank is two, choose
nonzero row types $r_1,r_2$ as a basis, and let $A$, $B$, and $C$ be the
classes with rows $r_1$, $r_2$, and $r_1+r_2$, respectively. Empty
classes are allowed. The neighborhood in $P$ of every vertex outside
$P$ is one of
\[
        \emptyset,\qquad A\cup C,\qquad B\cup C,\qquad A\cup B.
\]
At cutrank one, there is only one nonzero row class.
\end{lemma}

\begin{proof}
For $v\notin P$, set
$\alpha=r_1(v)$ and $\beta=r_2(v)$. The entries in the column of $v$
are $\alpha$ on $A$, $\beta$ on $B$, and $\alpha+\beta$ on $C$.
The four choices of $(\alpha,\beta)\in\GF_2^2$ give the displayed
neighborhoods. The rank-one statement is immediate.
\end{proof}

\begin{definition}[two-profile layout]\label{def:twoprofile}
A two-profile layout is a vertex order in which every prefix cut has at
most two distinct nonzero rows.
\end{definition}

\begin{proposition}[two profiles give width six]\label{prop:twoprofile}
If $G$ has a two-profile layout, then
\[
        \lsfw(G)\le6.
\]
\end{proposition}

\begin{proof}
Such a layout has active neighborhood-width at most two. Apply
Proposition~\ref{prop:layouttoplay}.
\end{proof}

\begin{proposition}[cycles]\label{prop:cycles}
For every $n\ge4$,
\[
        \lsfw(C_n)\le6.
\]
\end{proposition}

\begin{proof}
Order the vertices $v_1,\ldots,v_n$ cyclically. For
$2\le i\le n-2$, the only nonzero rows at
$P_i=\{v_1,\ldots,v_i\}$ are represented by $v_1$ and $v_i$.
The first sees only $v_n$ in the suffix, and the second sees only
$v_{i+1}$. Every vertex strictly between them in the order has zero
row. At $P_{n-1}$, the vertices $v_1$ and $v_{n-1}$ have the same
nonzero row, since both see only $v_n$. The initial prefixes also have
at most two nonzero rows. The cyclic order is therefore a two-profile
layout, and Proposition~\ref{prop:twoprofile} applies.
\end{proof}

The third row type at a rank-two cut is not merely an artifact of a
poorly chosen layout. The next graph has active neighborhood-width three
in every layout.

\begin{definition}[the graph $H$]\label{def:H}
Let $C$ be the cycle on $\mathbb Z_7$, with subscripts read modulo seven,
and let
\[
        H=\overline C.
\]
Thus $uv$ is an edge of $H$ exactly when the cyclic distance between
$u$ and $v$ is two or three.
\end{definition}

\begin{lemma}[profiles of $H$]\label{lem:Hprofiles}
The graph $H$ satisfies
\[
        \lrw(H)=2\qquad\text{and}\qquad\anw(H)=3.
\]
It has no pair of false twins or true twins.
\end{lemma}

\begin{proof}
For the order
\[
        \pi=(0,1,4,3,2,5,6),
\]
the cutranks and active row counts after successive prefixes are
\[
\begin{array}{c|rrrrrrr}
i&1&2&3&4&5&6&7\\
\hline
\rho_H(P_i)&1&2&2&2&2&1&0\\
|\mathcal A_i^\pi|&1&2&3&3&3&1&0.
\end{array}
\]
Thus $\lrw(H)\le2$ and $\anw(H)\le3$.

Let $X$ be any three-vertex set and put $Y=V(H)\setminus X$. Every
vertex of $X$ has a neighbor in $Y$, since its degree in $H$ is four.
We show that the three rows from $X$ to $Y$ are distinct. For
$x,x'\in X$, equality of their $H$-rows on $Y$ is equivalent to
\[
        N_C(x)\cap Y=N_C(x')\cap Y.
\]
Consider the cyclic distance between $x$ and $x'$. At distance one, the
two cycle neighborhoods have distinct outward vertices, and the third
member of $X$ cannot remove both. At distance two, their symmetric
difference consists of two vertices, and again the third member cannot
remove both. At distance three, the two cycle neighborhoods are disjoint.
Removing one further vertex cannot make their intersections with $Y$
equal. Thus every three-vertex prefix has three distinct nonzero rows,
and $\anw(H)\ge3$.

It follows that $\anw(H)=3$. Lemma~\ref{lem:nwcompare} gives
$3\le2^{\lrw(H)}-1$. Together with the displayed layout, this yields
$\lrw(H)=2$.

Finally,
\[
        N_H(u)=V(H)\setminus\{u,u-1,u+1\},
        \qquad
        N_H[u]=V(H)\setminus\{u-1,u+1\}.
\]
These open and closed neighborhoods are distinct for different vertices,
so $H$ has neither false twins nor true twins.
\end{proof}

\begin{theorem}[the affine estimate fails at rank two]\label{thm:H}
The graph $H$ satisfies
\[
        \lrw(H)=2,\qquad
        \lsfw(H)=8.
\]
Consequently, the inequality
\[
        \lsfw(G)\le2\,\lrw(G)+2
\]
does not hold for all graphs.
\end{theorem}

\begin{proof}
The upper bound follows from $\anw(H)=3$ and
Theorem~\ref{thm:calibration}. Suppose that a play of width at most seven
constructs $H$. Since $H$ is twin-free,
Lemma~\ref{lem:birthrigidity} forces every birth to be a singleton.

Consider the moment before the fourth birth, and let $X$ be the three
vertices already born. By Lemma~\ref{lem:Hprofiles}, their rows toward
the four unborn vertices are nonzero and pairwise distinct. Closed
vertices have zero rows, while vertices in one live cable have the same
future row. Hence the three vertices of $X$ are live in three separate
singleton cables.

Let $v$ be the fourth vertex born and put
$Y=V(H)\setminus(X\cup\{v\})$. Immediately after this birth, four cables
are live. A birth would cost nine, and a block or clique would cost
eight. Before any edge incident with $v$ can be placed, the next move
must therefore be a close or a fold.

No cable can close. The new vertex $v$ has a neighbor in $X$, since it
has four neighbors and $Y$ has only three vertices. The corresponding
edge has not been placed. For an old vertex $x\in X$, its row before the
birth was nonzero. If $xv$ is an edge, that edge is unplaced. If $xv$ is
not an edge, then $x$ has a neighbor in $Y$. In either case $x$ cannot
close.

Two old vertices cannot fold. If their rows on $Y$ differ, an unborn
vertex would distinguish them after the fold. If their rows on $Y$ are
equal, their distinct rows before the birth differ exactly at coordinate
$v$. Thus $v$ is adjacent to exactly one of them. After the fold, the
unplaced required edge cannot be added without also adding the forbidden
other edge.

It remains to rule out a fold of $v$ with an old vertex $x$. Such a fold
requires
\[
        N_C(v)\cap Y=N_C(x)\cap Y,
\]
because otherwise a future vertex distinguishes the folded pair. Let the
cyclic distance between $x$ and $v$ be $d$.

If $d=1$, equality forces the two outward cycle neighbors of $x$ and
$v$ to be the other two vertices of $X$. Taking the outward neighbor
$y$ of $x$, we have $vy\in E(H)$ but $xy\notin E(H)$. If $d=2$,
equality similarly forces the two noncommon cycle neighbors into
$X$. The outward neighbor $y$ of $x$ again satisfies
$vy\in E(H)$ and $xy\notin E(H)$. In either case the required edge
$vy$ is still unplaced, and after folding $v$ with $x$ it cannot be
added without also adding the forbidden edge $xy$. If $d=3$, the two
cycle neighborhoods are disjoint. Only two of their four vertices can
lie in $X\setminus\{x\}$, so their intersections with $Y$ cannot be
equal.

Every possible next move is impossible. This contradiction proves
$\lsfw(H)\ge8$, and the upper bound gives equality.
\end{proof}

Theorem~\ref{thm:H} shows that row dimension alone does not support the
previously suggested affine upper estimate. It also shows that the
width-six sufficient condition in Proposition~\ref{prop:twoprofile}
cannot cover every graph of linear rank-width two.

\section{Conclusion}\label{sec:conclusion}

Sequential split-free cable width is governed, within an additive two, by
the number of active neighborhood classes in a linear layout. The
calibration gives both a canonical construction and a certificate
extraction procedure. It also separates two sources of complexity.
Cutrank measures the dimension needed to describe a boundary, while a
cable play pays for the row types that must remain distinguishable.

The exact active-neighborhood comparison leaves a small additive problem.
For a fixed value $k=\anw(G)$, it is natural to determine when the
sequential width is $2k$, $2k+1$, or $2k+2$. The graph $H$ in
Section~\ref{sec:ranktwo} shows that the upper endpoint can occur when
$k=3$. It remains to characterize the graphs attaining each value.

The branching parameter also remains largely open. The tree construction
shows that it can be much smaller than its sequential restriction, but no
general comparison with rank-width is proved here.

\begin{problem}[branching lower bound]\label{prob:branching}
Does every graph $G$ satisfy
\[
        \rw(G)\le
        \left\lfloor\frac{\sfw(G)}2\right\rfloor?
\]
\end{problem}

A positive answer would complete the formal analogy between branching
expressions and rank-width on one side, and sequential plays and linear
rank-width on the other. A negative answer would identify a genuine
asymmetry between the two split-free constructions.

\bibliographystyle{fundam}
\bibliography{mybib}

\end{document}